\documentclass[11pt]{amsart}
\usepackage{amsmath,amssymb,amsthm,graphicx, url, verbatim}
\usepackage{color}
\usepackage[margin=1.25in]{geometry}

\theoremstyle{definition}
\newtheorem{defn}{Definition}[section]

\newtheorem{eg}[defn]{Example}
\theoremstyle{plain}
\newtheorem{thm}[defn]{Theorem}
\newtheorem{conj}[defn]{Conjecture}
\newtheorem{prop}[defn]{Proposition}
\newtheorem{lem}[defn]{Lemma}

\newcommand{\qb}[2]{\left[\begin{array}{c} #1 \\ #2  \end{array}\right]}
\newcommand{\s}{\sigma}
\newcommand{\Ga}{\Gamma}
\newcommand*{\pdash}{\rule[0.5ex]{1.5em}{1.5pt}}

\begin{document}
\title[Slopes for pretzel knots ]{Slopes for pretzel knots.}
\author[C. Lee]{Christine Ruey Shan Lee}
\author[R. van der Veen]{Roland van der
 Veen}

\address[]{Department of Mathematics, University of Texas, Austin TX 78712}
\email[]{clee@math.utexas.edu}

\address[]{Mathematisch Instituut, Leiden University, Leiden, Netherlands}
\email[]{r.i.van.der.veen@math.leidenuniv.nl}

\thanks{Lee was supported by NSF grant DMS 1502860. Van der Veen was supported by the Netherlands foundation for scientific research (NWO).}

\begin{abstract}Using the Hatcher-Oertel algorithm for finding boundary slopes of Montesinos knots, we prove the Slope Conjecture and the Strong Slope Conjecture for a family of 3-tangle pretzel knots. More precisely, we prove that the maximal degrees of the colored Jones polynomial of such knots determine a boundary slope as predicted by the Slope Conjecture, and that the linear term in the degrees correspond to the Euler characteristic of an essential surface. \end{abstract}
\maketitle

\section{Introduction}
Shortly after its invention the Jones polynomial was applied very successfully in knot theory. For example, it was the main tool in proving the Tait conjectures. After that many deeper connections to low-dimensional topology were uncovered while others remain conjectural and have little direct applications to questions in knot theory. With the Slope Conjecture, the Jones polynomial gives a new perspective on boundary slopes of surfaces in the knot complement. The conjecture provides many challenging and effective predictions about boundary slopes that cannot yet be attained by classical topology.

Precisely, the Slope Conjecture \cite{Gar11} states that the growth of the maximal degree
of $J_K(n;v)$ determines the boundary slope of an essential surface in the knot complement, see Conjecture \ref{conj:slopes}a. The conjecture has been verified for knots with up to 10 crossings \cite{Gar11}, alternating knots \cite{Gar11}, and more generally adequate knots \cite{FKP11, FKP13}.  Based on the work of \cite{DG12}, Garoufalidis and van der Veen proved the conjecture for 2-fusion knots \cite{GR14}. In \cite{KT15}, Kalfagianni and Tran showed that the set of knots satisfying the Slope Conjecture is closed under taking the $(p,q)$-cable with certain conditions on the colored Jones polynomial. They also formulated the Strong Slope Conjecture, see Conjecture \ref{conj:slopes}b, and verified it for adequate knots and their iterated cables, iterated torus knots, and a number of other examples. 

In this paper we prove the Slope Conjecture and the Strong Slope Conjecture for families of 3-string pretzel knots. This is especially interesting since many of the slopes found are non-integral. Our method is a comparison between calculations of the colored Jones polynomial based on knotted trivalent graphs and $6j$-symbols (called \emph{fusion} in \cite{GR14}), and the Hatcher-Oertel algorithm for Montesinos knots. Apart from providing more evidence for these conjectures, our paper is also a first step towards a more conceptual approach, which compares the growth of the degrees of the polynomial to data from curve systems on 4-punctured spheres.

The Slope Conjecture also provides an interesting way to probe more complicated questions such as the AJ conjecture \cite{FGL, GaChar}. According to the AJ conjecture, the colored Jones polynomial satisfies a q-difference equation that encodes the A-polynomial. The slopes of the Newton polygon of the A-polynomial are known to be boundary slopes of the knot \cite{CCGLS}. In this way the Slope Conjecture is closely related to the AJ conjecture \cite{Gar}. Of course the q-difference equation alone does not determine the colored Jones polynomial uniquely; in addition one would need to know the initial conditions or some other characterization. One way to pin down the polynomial would be to consider its degree and so one may ask: \emph{Which boundary slopes are selected by the the colored Jones polynomial?} We hope the present paper will provide useful data for attacking such questions.

We may also consider stabilization properties of the colored Jones polyomials such as heads and tails \cite{Arm}. Given an exact formula for the degree such as the one we write down, it is not hard to see what the tail looks like, but we do not pursue this in this paper.

\subsection{The Slope Conjectures}

For the rest of the paper, we consider a knot $K \subset S^3$. 

\begin{defn} An orientable and properly embedded surface $S \subset S^3 \setminus K$ is \emph{essential} if it is compressible, boundary-incompressible, and non boundary-parallel. If $S$ is non-orientable, then $S$ is \emph{essential} if its orientable double cover in $S^3\setminus K$ is essential in the sense defined above. 
\end{defn}

\begin{defn} Let $S$ be an essential and orientable surface with nonempty boundary in $S^3 \setminus K$. A fraction $p/q \in \mathbb{Q} \cup \{1/0\}$ is a \emph{boundary slope} of $K$ if $p\mu + q\lambda$ represents the homology class of $\partial S$ in $\partial N(K)$, where $\mu$ and $\lambda$ are the canonical meridian and longitude basis of $\partial N(K)$. The boundary slope of an essential non-orientable surface is that of its orientable double cover. 
\end{defn} 

The \emph{number of sheets}, $m$, of a properly embedded surface $S \subset S^3\setminus K$ is the number of times $\partial(S)$ intersects with the meridian circle of $\partial(N(K))$. 

For any $n\in \mathbb{N}$ we denote by $J_K(n;v)$ the unnormalized \emph{n-colored Jones polynomial} of $K$, see Section \ref{sec:CJPcomp}. Its value on the unknot is $\frac{v^{2n}-v^{-2n}}{v^{2}-v^{-2}}$
and the variable $v$ satisfies $v=A^{-1}$, where $A$ is the $A$-variable of the Kauffman bracket. Denote by $d_+ J_K(n)$ the maximal degree in $v$ of $J_K(n)$. Our variable $v$ is
the fourth power of that used in \cite{KT15}, thus absorbing superfluous factors of $4$. 

As a foundation for the study of the degrees of the colored Jones polynomial we apply the main result of \cite{GL05}
that says that the sequence of polynomials satisfies a $q$-difference equation (i.e. is $q$-holonomic). Theorem 1.1 of \cite{Gar11b} then implies that the degree
must be a quadratic quasi-polynomial, which may be formulated as follows.  

\begin{thm}\cite{Gar11b}\label{thm:quasip}\\
For every knot $K$ there exist integers $p_K,C_K \in \mathbb{N}$ and quadratic polynomials $Q_{K,1}\dots Q_{K,p_K}\in \mathbb{Q}[x]$ such that for all $n>C_K$ ,
\[d_+ J_K(n) = Q_{K,j}(n) \quad\text{if}\quad n = j\ (\mathrm{mod}\ p_K).\]
\end{thm}

The slope conjectures predict that the coefficients of the polynomials $Q_{K,j}$ have a direct topological interpretation. 

\begin{conj} \label{conj:slopes} \ 
If we set $Q_{j,K}(x) = a_jx^2+2b_jx+c_j$, then for each $j$ there exists an essential surface $S_j \subset S^3 \setminus K$ such that
\begin{enumerate}
\item[a.](Slope Conjecture  \cite{Gar11})
$a_j$ is the boundary slope of $S_j$, 
\item[b.](Strong Slope Conjecture \cite{KT15})
Writing $a_j = \frac{x_j}{y_j}$ as a fraction in lowest terms we have 
$b_jy_j = \frac{\chi(S_j)}{|\partial S_j|}$,
where $\chi(S_j)$ is the Euler characteristic of $S_j$ and $|\partial S_j|$ is the number of boundary components.  
\end{enumerate}
\end{conj}

The numbers $a_j$ are called the \emph{Jones slopes} of the knot $K$. Our formulation of the Strong Slope Conjecture is a little sharper than the original. According to the formulation in \cite{KT15}, the surface $S_j$ may be replaced with $S_i$ for some $1\leq i\leq p_K$ not necessarily equal to $j$. 
For all examples known to the authors, the polynomials $Q_{K,j}$ all have the same leading term, so it is not yet possible to decide which is the correct statement. 

For completeness sake one may wonder about the constant terms $Q_{K,j}(0)$. It was speculated by Kalfagianni and the authors that perhaps we have: $Q_{K,j}(1)=0$ for each $j$. This surely holds in simple cases where one may take $p_K=1, C_K=0$, but not for the more complicated pretzel knot cases we will describe. Perhaps the constant term does have a topological interpretation that extends the slope conjectures further.

\subsection{Main results} \label{subsec:main}

Recall that a Montesinos knot $K(\frac{p_1}{q_1},\frac{p_2}{q_2}, \ldots,\frac{p_n}{q_n})$ is a sum of rational tangles \cite{conway}. As such both the colored Jones polynomial and the boundary slopes are more tractable than for general knots yet still highly non-trivial. 
When it is put in the \emph{standard form} as in Figure \ref{fig:stdform}, a Montesinos knot is classified by ordered sets of fractions 
$\left(\frac{\beta_1}{\alpha_1} \mod 1, \ldots, \frac{\beta_r}{\alpha_r} \mod 1 \right)$, considered up to cyclic permutation and reversal of order \cite{Bon79}. 
Here $e$ is the number indicated below when the Montesinos knot is put in the standard form as shown in Figure \ref{fig:stdform}. 
\begin{figure}[ht]
\includegraphics[scale=.7]{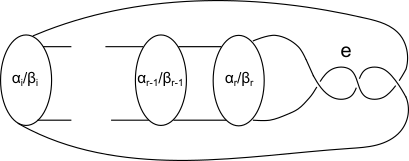}
\caption{A Montesinos knot in standard form. \label{fig:stdform}}
\end{figure} 

Moreover, a Montesinos knot is semi-adequate if it has more than 1 positive tangles or more than 1 negative tangles \cite{LT88}.   
Since the slope conjectures were settled for semi-adequate knots \cite{FKP11, FKP13, KT15}, we may restrict our attention to Montesinos knots with
exactly one negative tangle. The simplest case for which the results are not known are when there are three tangles in total.
For convenience we make further assumptions on the shape of the tangles. First we require the fractions to be $(\frac{1}{r}, \frac{1}{s}, \frac{1}{t})$, so that our knot is a pretzel knot
written $P(\frac{1}{r}, \frac{1}{s}, \frac{1}{t})$, and we assume $r<0<s,t$. An example of it is shown in Figure \ref{fig:pretzel}. For technical reasons, we restrict our family of pretzel knots a little more 
so that we can obtain the following result: 

\begin{thm}
\label{thm:slope} Conjecture \ref{conj:slopes} is true for the pretzel knots $P(\frac{1}{r}, \frac{1}{s},\frac{1}{t})$
where $r < -1 < 1 < s, t$, and $r, s, t$ odd in the following two cases: 
\begin{enumerate}
\item $2|r|<s,t$
\item  $|r|>s$ or $|r|>t$.
\end{enumerate}
\end{thm}

\begin{eg} For the knot $K=P(\frac{1}{-5}, \frac{1}{5}, \frac{1}{3} )$, the first three colored Jones polynomials are
\[
J_K(1;v)=1,
\]
\[
J_K(2;v)= v^{-34} + v^{-26} - v^{-22} - v^{-14} - v^{-10} + 2 v^2 + v^{10},
\]
\[
J_K(3;v)=v^{-100} + v^{-88} - v^{-84} - 2v^{-80} + v^{-76} - 3v^{-68} + 2v^{-60} - 
v^{-52} + v^{-48}+ \]
\[ 2v^{-44} + 3v^{-32} - v^{-24} + v^{-20} - v^{-16} - 2v^{-12} 
- v^{-8} + v^{-4} - v^4 + v^{12} - v^{20} + v^{24} + 2 v^{28}.
\]
In this case $p_K = 3$, notice the $2$ as a leading coefficient, this occurs for any $n$ divisible by $3$.
A table of the maximal degree of the first 13 colored Jones polynomials is more informative:
\begin{center}
\begin{tabular}{ c|*{13}{c} } 
 \hline
$n$ & 1  & 2  & 3  & 4  & 5   & 6   & 7   & 8   & 9   & 10  & 11  & 12  & 13 \\ 
\hline
$d_+J_K(n;v)$ & 0  & 10 & 28 & 62 & 104 & 154 & 220 & 294 & 376 & 474 & 580 & 694 & 824 \\  
 \hline
\end{tabular}
\end{center}
When $n=0 \mod 3$, the maximal degree $d_+J_K(n) = \frac{16}{3}n^2-6n-2$, and otherwise $d_+J_K(n) = \frac{16}{3}n^2-6n+\frac{2}{3}$.
So $a_j = \frac{16}{3}, b_j=-3$ and $c_0 = -2$, while $c_1,c_2 = \frac{2}{3}$.  

All these are matched by an essential surface of boundary slope $16/3$, 
a single boundary component, $3$ sheets, and Euler characteristic $-9$. 

\begin{figure}[ht]
\centering 
\includegraphics[scale=.2]{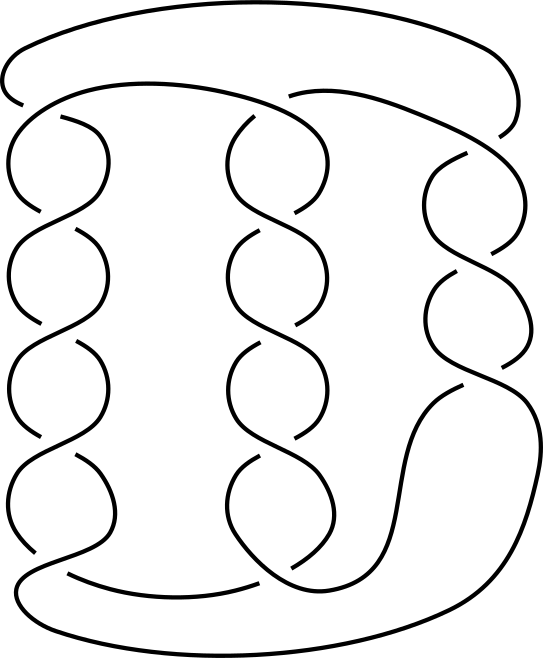}
\caption{\label{fig:pretzel} Pretzel knot $P(-\frac{1}{5}, \frac{1}{5}, \frac{1}{3})$.}

\end{figure} 
\end{eg}

The proof of our theorem follows directly from the two theorems below. The first dealing with the colored Jones polynomial and the second with essential surfaces.

\begin{thm} \label{thm:slope1} Assume $r,s,t$ are odd,  $r<-1<1<s,t$, and $K=P(\frac{1}{r},\frac{1}{s},\frac{1}{t})$.
\begin{enumerate}
\item When $2|r| < s, t$ we have $p_K=1$ and $Q_{K,1}(n) = -2n+2$.
\item When $|r|>s$ or $|r|>t$ we have $p_K = \frac{-2+s+t}{2}$ and \[Q_{K,j} = 2\left(\frac{(1-st)}{-2+s+t} -r\right)n^2+2(2+r)n+c_j,\]
where $c_j$ is defined as follows. Assuming $0\leq j < \frac{-2+s+t}{2}$ set $v_j$ to be the (least) odd integer nearest to $\frac{2(t-1)j}{-2+s+t}$. Then 
\[c_j = \frac{-6 + s + t}{2} - \frac{2 j^2 (t-1)^2}{-2 + s + t}+2 j (t-1)v_j+\frac{2 - s - t}{2}v_j^2.\]
\end{enumerate}
\end{thm}

\begin{thm} \label{thm:slope2} Under the same assumptions as the previous theorem:
\begin{enumerate}
\item When $2|r| < s, t$ there exists an essential surface $S$ of $K$ with boundary slope $0=\frac{0}{1}$, and
\[\frac{\chi(S)}{|\partial S|} = -1 .\] 
\item When $|r| > s$ or $|r| > t$ there exists an essential surface with boundary slope $2\left(\frac{(1-st)}{-2+s+t} -r\right) = \frac{x_j}{y_j}$ (reduced to lowest terms), and 
\[ \frac{\chi(S)}{y_j \cdot |\partial S|} = 2+r.\] 
\end{enumerate}  
\end{thm} 

The exact same proofs work when $r$ is even and $s, t$ are odd. In other cases additional complications may arise. 
Coming back to the interpretation of the constant terms $c_j$, the above expressions make it clear that they cannot be determined by $a_j$ and $b_j$ alone. It seems an interesting challenge to
find a topological interpretation of the $c_j$. For more complicated knots it is likely (but unknown) that the periodic phenomena that we observe in the $c_j$ will also occur in the coefficients $a_j,b_j$.  

The main idea of the proof of Theorem \ref{thm:slope1} is to write down a state sum and consider the maximal degree of each summand in the state sum. If one is lucky only one single term in the state sum will have maximal degree. In that case the maximal degree of that summand is the maximal degree of the whole sum. The maximal degree of each term happens to be a piecewise quadratic polynomial, so the problem comes down to maximizing a polynomial over the lattice points in a polytope. As soon as there are multiple terms attaining the maximum things get more complicated. This is the reason for not considering all pretzel knots or even Montesinos knots. Similar results can be obtained at least for the remaining pretzel knots with 3 tangles, but not without considerable effort to control the potential cancellations between terms. Different tools are needed to give a satisfactory proof of the general case.

For Theorem \ref{thm:slope2}, the Hatcher-Oertel algorithm works in more general settings. However, exhibiting a surface with the specified Euler characteristic and boundary components may not be so simple in general.

The organization of the paper is as follows: In Section \ref{sec:CJPcomp}, we describe the computation for the degree of the colored Jones polynomial, which will determine an exact formula for its degrees and prove Theorem \ref{thm:slope1}. We describe the Hatcher-Oertel algorithm as it suits our purpose in Section \ref{sec:HO}. In Section \ref{sec:verifying}, we prove Theorem \ref{thm:slope2} by applying the algorithm and describing the boundary slopes corresponding to the Jones slopes. Finally, possible generalizations are discussed in Section \ref{sec:discussion}. 
\newline

\noindent {\bf Acknowledgement.} We would like to thank Stavros Garoufalidis, Efstratia Kalfagianni and Anh Tran for several stimulating conversations, as well as the organizers at KIAS for providing excellent working conditions during the First Encounter to Quantum Topology: School and Workshop Conference in Seoul, Korea.

\section{Colored Jones polynomial} \label{sec:CJPcomp}

In this section we define the colored Jones polynomial, give an example of how it can be computed, and give a lower bound for its maximal degree.

\subsection{Definition of colored Jones polynomial using Knotted Trivalent Graphs}

Knotted trivalent graphs (KTGs) provide a generalization of knots that is especially suited for introducing the colored Jones polynomial in an intrinsic way.

\begin{defn} \
\begin{enumerate}
\item A framed graph is a 1-dimensional simplicial complex $\Gamma$ together with an embedding $\Gamma\to \Sigma$ of $\Gamma$ into a surface with boundary $\Sigma$ as a spine.
\item A coloring of $\Gamma$ is a map $\sigma:E(\Gamma)\to \mathbb{N}$, where $E(\Gamma)$ is the set of edges of $\Gamma$. 
\item A Knotted Trivalent Graph (KTG) is a a trivalent framed graph embedded (as a surface) into $\mathbb{R}^3$, considered up to isotopy.
\end{enumerate}
\end{defn}

\begin{figure}[ht]
\begin{center}
\includegraphics[width=\linewidth]{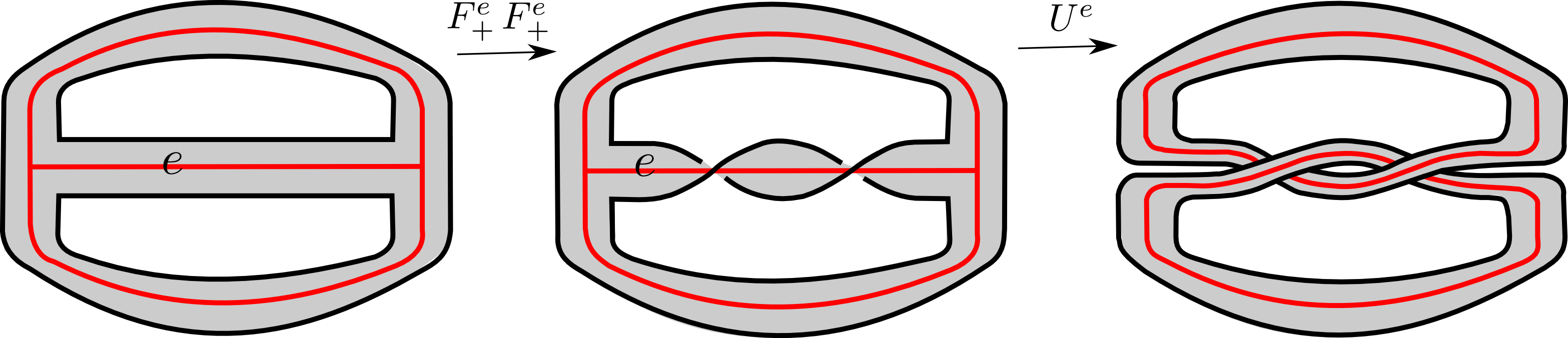}
\end{center}
\caption{\label{fig.Hopf} 
The KTG $\Theta$ (left) and the Hopf link as a KTG (right). Two framing changes followed by an unzip on the middle edge turn the Theta into the Hopf link.}
\end{figure}

A fundamental example of a KTG is the planar theta graph $\Theta$ shown in Figure \ref{fig.Hopf} on the left. It has two vertices and three edges that are embedded in the two holed disk. Framed links are special cases of KTGs with no vertices, see for example the Hopf link $H$ in Figure \ref{fig.Hopf} on the right. The reason we prefer the more general set of KTGs is the rich 3-dimensional operations that they support. In the figure we see an example of how the link $H$ arises from the theta graph by simple operations that are described in detail below.

The first operation on KTGs is called a \emph{framing change} denoted by $F^e_{\pm}$. It cuts the surface $\Sigma$ transversal to an edge $e$, rotates one side by $\pi$
and reglues. The second operation is called \emph{unzip}, $U^e$. It doubles a chosen edge along its framing, deletes its end-vertices and
joins the result as shown in Figure \ref{fig.KTGMoves}. The final operation is called $A^w$ and expands a vertex $w$ into a triangle as shown in Figure \ref{fig.KTGMoves}.
The result after applying an operation $M$ to KTG $\Ga$ will be denoted by $M(\Ga)$. For example, the Hopf link can be presented as $U^e(F^e_+(F^e_+(\Theta)))$.

\begin{figure}[ht]
\begin{center}
\includegraphics[width=\linewidth]{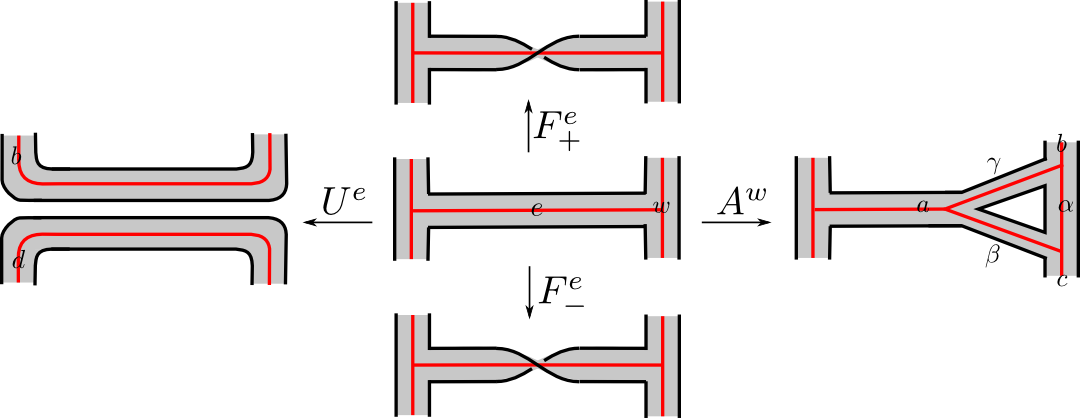}
\end{center}
\caption{\label{fig.KTGMoves} Operations on Knotted Trivalent Graphs: framing change $F_{\pm}$, unzip $U$, and triangle move $A$ applied to an edge $e$ and vertex $w$ shown in the middle.}
\end{figure}

These operations suffice to produce any KTG from the theta graph as was shown by D. Thurston \cite{ThuKTG}, see also \cite{Vdv}.
\begin{prop}
Any KTG can be generated from $\Theta$ by repeatedly applying the three operations $F_\pm,U$ and $A$ defined above.
\end{prop}

In view of this result, the colored Jones polynomial of any KTG is determined once we fix the value of any colored theta graph and describe how it
changes when applying any of the KTG operations.

\begin{defn}
The colored Jones polyomial of a KTG $\Gamma$ with coloring $\s$, notation $\langle \Ga,\s\rangle$, is defined by the following four equations explained below.
\begin{equation}\langle\Theta\ a,b,c\rangle = O^{\frac{a+b+c}{2}}\qb{\frac{a+b+c}{2}}{\frac{-a+b+c}{2},\frac{a-b+c}{2},\frac{a+b-c}{2}},\end{equation}
\begin{equation}
\langle F^e_{\pm}(\Ga),\s \rangle = f(\s(e))^{\pm1}\langle \Ga,\s \rangle,
\end{equation}
\begin{equation}
\langle U^e(\Ga),\s \rangle = \langle \Gamma,\s \rangle\sum_{\s(e)}\frac{O^{\s(e)}}{\langle\Theta\ \s(e),\s(b),\s(d)\rangle}, 
\end{equation} and
\begin{equation}
\langle A^w(\Ga),\s \rangle = \langle \Ga,\s \rangle\Delta(a,b,c,\alpha,\beta,\gamma).
\end{equation}
As noted above, a $0$-framed knot $K$ is a special case of a $KTG$. In this case we denote its colored Jones polynomial by $J_K(n+1) = (-1)^n\langle K,n \rangle$, where $n$ means the single edge has color $n$. The extra minus sign is to normalize the unknot as $J_O(n) = [n]$.
\end{defn}

To explain the meaning of each of these equations we first set $[k] =\frac{v^{2k}-v^{-2k}}{v^2-v^{-2}}$ and $[k]! = [1][2]\dots[k]$ for $k\in \mathbb{N}$ and $[k]! = 0$ if $k\notin \mathbb{N}$.
Now the symmetric multinomial coefficient is defined as:
\[\qb{a_1+a_2+\dots+a_r}{a_1,a_2,\dots, a_r} = \frac{[a_1+a_2+\dots+a_r]!}{[a_1]!\dots[a_r]!}.\]
In terms of this, the value of the $k$-colored ($0$-framed) unknot is $O^k = (-1)^k[k+1] = \langle O, k\rangle$, and the above formula for the theta graph whose edges are colored $a,b,c$ includes a quantum trinomial.
Next we define 
\begin{align*}
& \Delta(a,b,c,\alpha,\beta,\gamma) = \\ 
& \sum_{z}\frac{(-1)^z}{(-1)^{\frac{a+b+c}{2}}} \qb{z+1}{\frac{a+b+c}{2}+1}\qb{\frac{-a+b+c}{2}}{z-\frac{a+\beta+\gamma}{2}}
\qb{\frac{a-b+c}{2}}{z-\frac{\alpha+b+\gamma}{2}}\qb{\frac{a+b-c}{2}}{z-\frac{\alpha+\beta+c}{2}}.
\end{align*}
The formula $\Delta$ is the quotient of the $6j$-symbol and a theta, the summation range for $\Delta$ is finite as dictated by the binomials. Finally, we define 
\[f(a) = i^{-a}v^{\frac{-a(a+2)}{2}}.\]
This explains all the symbols used in the above equations. In the equation for unzip the sum is taken over all possible colorings of the edge $e$ that was unzipped. All other edges
are supposed to have the same color before and after the unzip. Again this results in a finite sum since the only values that may be non-zero are when $\s(e)$ is between $|\s'(b)-\s'(d)|$ and $\s'(b)+\s'(d)$ and has the same parity. Finally in the equation for $A$,  the colors of the six edges involved in the $A$ operation are denoted $a,\alpha,b,\beta,c,\gamma$ as shown in Figure \ref{fig.KTGMoves}.

The above definition agrees with the integer normalization used in \cite{Cost}. It was shown there that $\langle \Ga, \s \rangle$ is a Laurent polynomial in $v$ and does not depend on the
choice of operations we use to produce the KTG. As a relatively simple example,  the reader is invited to verify that the colored Jones polynomial of the Hopf link $H$ whose components are colored $a,b$
is given by the formula $\langle H, a,b \rangle = \sum_c f(c)^2 \frac{O^c}{\langle \Theta\ a,b,c \rangle}\langle{\Theta\ a,b,c}\rangle = (-1)^{a+b}[(a+1)(b+1)].$
The first equality sign follows directly from reading Figure \ref{fig.Hopf} backwards.

The above definition may appear a little cumbersome at first sight, but it is more three-dimensional and less dependent on knot diagrams and produces concise formulas for Montesinos knots.
For example,  the colored Jones polynomial of the $0$-framed Pretzel knot is given in the following lemma.  

\begin{lem}
\label{lem.CJP}
For $r,s,t$ odd, the colored Jones polynomial of the pretzel knot $P(\frac{1}{r},\frac{1}{s},\frac{1}{t})$ is given by 
\begin{align*}
& J_{P(\frac{1}{r},\frac{1}{s},\frac{1}{t})}(n+1) = \\ 
& (-1)^n\sum_{a,b,c}\frac{O^aO^bO^cf(a)^rf(b)^sf(c)^t\langle \Theta\ a,b,c\rangle}{\langle\Theta\ a,n,n\rangle \langle\Theta\ b,n,n\rangle \langle\Theta\ c,n,n\rangle}\Delta(a,b,c,n,n,n)^2.
\end{align*}
Here the sum is over all even $0 \leq a,b,c\leq 2n$ 
that satisfy the triangle inequality (this comes from $\langle \Theta\ a,b,c\rangle$). Also each non-zero term in the sum has leading coefficient $C(-1)^{\frac{ar+bs+ct}{2}}$ for some $C\in \mathbb{R}$ independent of $a,b,c$.
\end{lem}
\begin{proof}
The exact same formula and proof works for general $r,s,t$ except that we only get a knot when at most one of them is odd and have to correct the framing by adding the term
$f(n)^{-2\mathrm{Wr}(r,s,t)-2r-2s-2t}$ where the writhe is given by $\mathrm{Wr}(r,s,t)=-(-1)^{rst}((-1)^r r+(-1)^s s + (-1)^t t)$.
In Figure \ref{fig.PretzelKTG} we illustrate the proof for the pretzel knot $K=P(\frac{1}{3},\frac{1}{1},\frac{1}{2})$, the general case is similar.
The first step is to generate our knot $K$ from the theta graph by KTG moves. One way to achieve this is shown in the figure. To save space we did not explicitly draw the framed bands
but instead used the blackboard framing. The dashes indicate half twists when blackboard framing is not available or impractical. The exact same sequence of moves will produce any pretzel knot, one just needs to adjust the number of framing changes accordingly. Note that the unzip applied to a twisted edge produces two twisted bands that form a crossing. This is natural considering that the black lines stand for actual strips.
Reading backwards and applying the above equations, we may compute $J_K(n+1)$ as follows. The unzips yield three summations, the framing change multiplies everything by the factors $f$, the $A$ moves both yield the same labeling and hence a $\Delta^2$ and the final theta completes the formula.

To decide the leading coefficient of the terms in the sum corresponding to $0\leq a,b,c\leq 2n$ we see that the unknots contribute $(-1)^{a+b+c}$, the $f$-terms multiply this by $i^{ar+bs+ct}$ and something independent of $a,b,c$ and the thetas contribute $(-1)^{a+b+c+3n}$. The $\Delta^2$ must have leading coefficient $1$.

\begin{figure}[ht]
\begin{center}
\includegraphics[width=\linewidth]{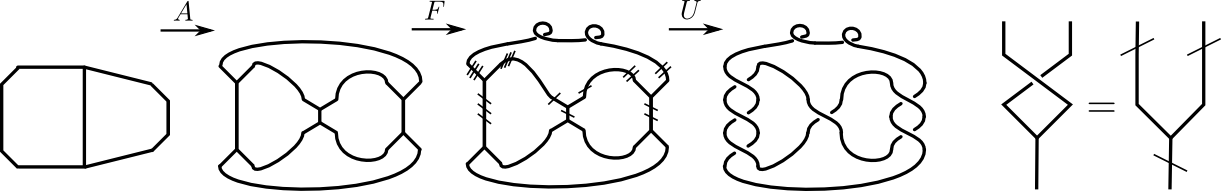}
\end{center}
\caption{\label{fig.PretzelKTG}
Starting from a theta graph (left), we first apply the $A$-move to both vertices, next change the framing on many edges (half twists in the edge bands are denoted by a dash), and finally unzip the vertical edges to obtain a $0$-framed diagram for the pretzel knot $P(\frac{1}{3},\frac{1}{1},\frac{1}{2})$. The crossings arise from the half twists using the isotopy shown on the far right.}
\end{figure}
\end{proof}

\subsection{The degree of the colored Jones polynomial}

Now that we defined the colored Jones polynomial of a KTG and noted that it is always a Laurent polynomial, we may consider its maximal degree.

\begin{defn}
Denote by $d_+\langle \Ga,\s \rangle$ the highest degree in $v$ of $\langle \Ga, \s \rangle$.
\end{defn}

The highest order term in the four equations defining the colored Jones polynomial of KTGs yields a lot of information on the behaviour of $d_+$.
We collect this information in the following lemma whose proof is elementary.

\begin{lem} \label{lem:destimate}
\begin{equation}
d_+\langle \Theta \ a,b,c\rangle = a(1-a)+b(1-b)+c(1-c)+\frac{(a+b+c)^2}{2}, 
\end{equation}
\begin{equation}
d_+\langle F^e_{\pm}(\Ga),\s \rangle = \pm d_+f(\s(e))\langle \Ga,\s \rangle, 
\end{equation}
\begin{equation}
d_+\langle U^e(\Ga),\s \rangle \geq d_+\langle \Gamma,\sƒ \rangle + \max_{\s(e)} d_+O^{\s(e)}-d_+\langle\Theta\ \s(e),\s(b),\s(d)\rangle, 
\end{equation}
and
\begin{equation}
d_+\langle A^w(\Ga),\s \rangle = d_+\langle \Ga,\s \rangle+d_+\Delta(a,b,c,\alpha,\beta,\gamma). 
\end{equation}
\end{lem}
Here $d_+f(a) = -a(a+2)/2$ and $d_+O(a) = 2a$. The maximum is taken over $|\s(b)-\s(d)|\leq \s(e)\leq \s(b)+\s(d)$. Note the inequality sign, since we cannot guarantee the leading terms will not cancel out. However for the inequality to be strict, \emph{at least two} term have to attain the maximum.
Finally,
\begin{align}
d_+\Delta(a,b,c,\alpha,\beta,\gamma) &= g(m+1,\frac{a+b+c}{2}+1)+g(\frac{-a+b+c}{2},m-\frac{a+\beta+\gamma}{2}) \notag \\
&+g(\frac{a-b+c}{2},m-\frac{\alpha+b+\gamma}{2})+g(\frac{a+b-c}{2},m-\frac{\alpha+\beta+c}{2}),  
\end{align}
where $g(n,k) = 2k(n-k)$ and $2m = a+b+c+\alpha+\beta+\gamma-\max(a+\alpha,b+\beta,c+\gamma)$

Applying Lemma \ref{lem:destimate} to the formula of Lemma \ref{lem.CJP} for pretzel knots yields the following theorem:

\begin{thm}
\label{thm.JonesDeg}
Suppose $r,s,t$ are odd,  
\[
d_+J_{P(\frac{1}{r},\frac{1}{s},\frac{1}{t})}(n) = \begin{cases} -2n+2, &\mbox{if } s,t>-2 r\\
 2(\frac{1-st}{-2 + s + t}-r)n^2+  2(2 + r)n+c_n, &\mbox{if } s<-r \mbox{ or }  t < -r \end{cases}\]
where $c_n$ is defined as follows. Let $0\leq j < \frac{-2+s+t}{2}$ be such that $n = j \mod \frac{-2+s+t}{2}$ and
set $v_j$ to be the (least) odd integer nearest to $\frac{2(t-1)j}{-2+s+t}$. Then 
\[c_n = \frac{-6 + s + t}{2} - \frac{2 j^2 (t-1)^2}{-2 + s + t}+2 j (t-1)v_j+\frac{2 - s - t}{2}v_j^2\]
\end{thm}
\begin{proof}
The domain of summation $D_n$ is the intersection of the cone $|a-b|\leq c \leq a+b$ with
the cube $[0,2n]^3$ with the lattice $(2\mathbb{Z})^3$, so the maximal degree of the summands gives rise to the following inequality: 
\[
d_+J_{P(\frac{1}{r},\frac{1}{s},\frac{1}{t})}(n+1) \leq \max_{a,b,c\in D_n}\Phi(a,b,c,n),\]
where
\begin{align*}
\Phi(a,b,c,n)&=\ d_+O^a+d_+O^b+d_+O^c+d_+f(a)r+d_+f(b)s+d_+f(c)t+2d_+\Delta(a,b,c,n,n,n)\\
&+\ d_+\langle\Theta\ a,b,c\rangle-d_+\langle\Theta\ a,n,n\rangle-d_+\langle\Theta\ b,n,n\rangle-d_+\langle\Theta\ c,n,n\rangle.
\end{align*}

In general, this is just an inequality, but when $\Phi$ takes a \emph{unique} maximum, no cancellation can occur so
we have an actual equality.

To analyse the situation further we focus on the case of interest, which is $r \leq -1 < 2 \leq s, t$ all odd.
In that case we have the following three inequalities on $D_n$: $\Phi(a+2,b,c,n)>\Phi(a,b,c,n)$ and $\Phi(a,b+2,c,n)<\Phi(a,b,c,n)$ and $\Phi(a,b,c+2,n)<\Phi(a,b,c,n)$.
This shows that the maxima on $D_n$ must occur when $a=b+c$ and so we may restrict our attention to the triangle $T_n$ given by $0\leq b,c, b+c\leq 2n$.
On $T_n$ we compute
\begin{align*}
R(b,c) &= \Phi(b+c,b,c,n) \\
&=  - \frac{(r+s)b^2}{2} -(1+r) b c - \frac{(r+t)c^2}{2}+(2-r-s) b + (2-r-t) c- 2 n.
\end{align*}

With stronger assumptions,  we easily find many cases where $R(b, c)$ has a unique maximum on $T_n$:

First, if $s,t>-2 r$ then $R(b+2,c)<R(b,c)$ and $R(b,c+2)<R(b,c)$,  so any maximum must be at the origin $b=c=0$.
Secondly, if $s < -r$ or $t < -r$ then $R(b+2,c)>R(b,c)$ and $R(b,c+2)>R(b,c)$, so any maximum must be on the line $b+c=2n$.

In the first case we have $R(0,0) = -2n$,  so 
\[d_+J_{P(\frac{1}{r},\frac{1}{s},\frac{1}{t})}(n) = -2(n-1).\] 

In the second case we see that $R(b,2n-b)$ is a negative definite quadratic whose (real) maximum is at $m=\frac{2 n (t-1) -s+t}{-2 + s + t}$ and $0\leq m\leq 2n$ since $s>1$.
If $m$ is an odd integer,  then there are precisely two maxima and they may cancel out if the coefficients of the leading terms are opposite.
From Lemma \ref{lem.CJP} we know that the leading coefficients are $C(-1)^{\frac{ar+bs+ct}{2}}$ for some constant $C$ independent of $a,b,c$. On the diagonal $a=b+c$ and $c=2n-b$
we see that no cancellation will occur since $s+t$ is even.

Define $m'$ to be $m$ rounded down to the nearest even integer, then the exact maximal degree will be given by
\[d_+J_{P(\frac{1}{r},\frac{1}{s},\frac{1}{t})}(n+1) = R(m',2n-m')\]
To get an exact expression we set $N = n+1$ and $N=q \frac{-2+s+t}{2}+j$ for some $0\leq j < \frac{-2+s+t}{2}$.
Now $m = \frac{2(t-1)N}{-2+s+t}-1 = (t-1)q-1+\frac{2(t-1)j}{-2+s+t}$ so $m' = 2(t-1)\frac{N-j}{-2+s+t}-1+v_j$ where $v_j$ is the (least) odd integer nearest to $\frac{2(t-1)j}{-2+s+t}$. 
Finally expanding $R(m',2(N-1)-m')$ as a quadratic in $N$ we find the desired expression for $d_+J_{P(\frac{1}{r},\frac{1}{s},\frac{1}{t})}(N)$

\end{proof}

The technique presented here can certainly be strengthened and perhaps be extended to more general pretzel knots, Montesinos knots and beyond. However,
serious issues of possible and actual cancellations will continue to cloud the picture. More conceptual methods need to be developed.

\section{Boundary slopes of 3-string Pretzel knots.} \label{sec:HO}
In this section we describe the Hatcher-Oertel Algorithm \cite{HatcherOertel} as restricted to pretzel knots $P(\frac{1}{r}, \frac{1}{s}, \frac{1}{t})$. In an effort to make this paper self-contained, we describe explicitly how one may associate a candidate surface to an edgepath system, and how to compute boundary slopes and the Euler characteristic of an essential surface corresponding to an edgepath system. Readers who are familiar with the algorithm may skip to Section \ref{sec:verifying} directly. Our exposition follows that of \cite{KS10} and \cite{Ich15}. Dunfield has implemented the algorithm completely in a program \cite{DunfieldM}, which determines the list of boundary slopes given any Montesinos knot. For other examples of applications and expositions of the algorithm, see \cite{KS07}, \cite{ES05}.

\subsection{Incompressible surfaces and edgepaths.} \label{subsec:surface}
Viewing $S^3$ as the join of two circles $C_1$ and $C_2$, subdivide $C_2$ as an $n+1$-sided polygon. The join of $C_1$, called the \emph{axis}, with the $i$th edge of $C_2$ is then a ball $B_i$. For a Montesinos knot $K(\frac{p_1}{q_1}, \frac{p_2}{q_2}, \ldots, \frac{p_n}{q_n})$, we choose $B_i$ so that each of them contains a tangle of slope $p_i/q_i$, with $B_0$ containing the trivial tangle. These $n+1$ balls $B_i$ cover $S^3$, meeting each other only in their boundary spheres. We may view each tangle $p_i/q_i$ via a 2-bridge knot presentation in $S_i^2 \times [0, 1]$ in $B_i$, with the two bridges puncturing the 2-sphere at each $\ell \in [0, 1]$, and arcs of slope $p_i/q_i$ lying in $S_i^2 \times 0$. See \cite[Pg. 1, Figure 1b)]{HT85}.

We identify $S_i^2 \times \ell \setminus K$ with the orbit space $\mathbb{R}^2/\Gamma$, where $\Gamma$ is the group generated by 180$^\circ$ rotation of $\mathbb{R}^2$ about the integer lattice points. We use this identification to assign \emph{slopes} to arcs and circles on the four-punctured sphere $S_i^2 \times \ell \setminus K$ as in \cite{HT85}. \emph{Note.} This slope is not the same as the boundary slope of an essential surface! 

Hatcher and Thurston showed \cite[Theorem 1]{HT85} that every essential surface may be isotoped so that the critical points of the height function of $S$ in $B_i$ lie in $S_i^2\times \ell$ for distinct $\ell$'s, and the intersection consists of arcs and circles. Going from $\ell=0$ to $\ell=1$, the slopes of arcs and circles of $S\cap S_i^2 \times \ell$ at these critical levels determine an \emph{edgepath} for $B_i$ in a 1-dimensional cellular complex $\mathcal{D} \subset \mathbb{R}^2$. 

We may represent arcs and circles of certain slopes on a 4-punctured sphere via the $(a, b, c)-$coordinates as shown in Figure \ref{fig:curve}, where $c$ is parallel to the axis. The complex $\mathcal{D}$ is obtained by splicing a 2-simplex in the projective lamination space of the 4-punctured sphere in terms of projective weights $a$, $b$, and $c$, so that each point has horizontal coordinate $b/(a+b)$ and vertical coordinate $c/(a+b)$ in $\mathbb{R}^2$. 

\begin{figure}[ht]
 \def\svgwidth{.3\textwidth}
\includegraphics[width=.3\textwidth]{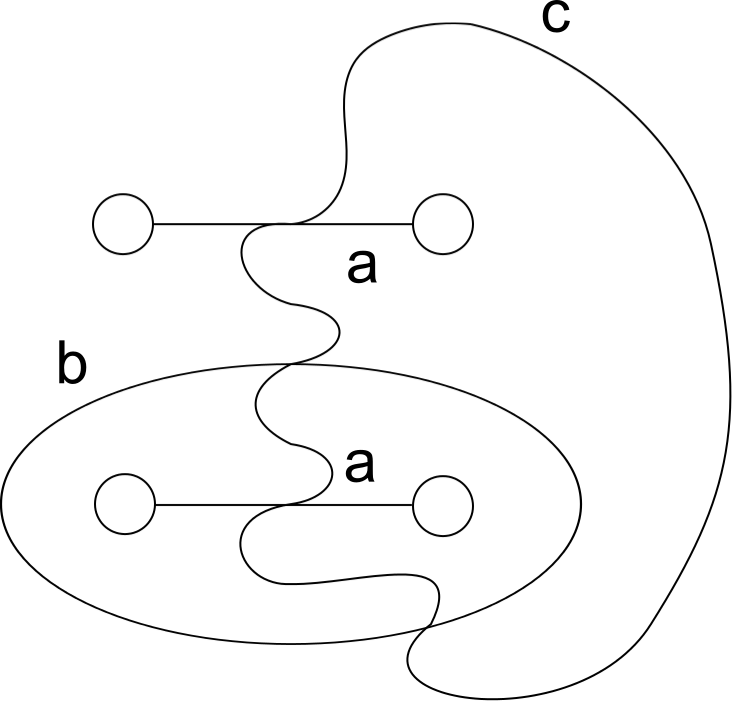}
\hspace{1cm}
\includegraphics[width=.3\textwidth]{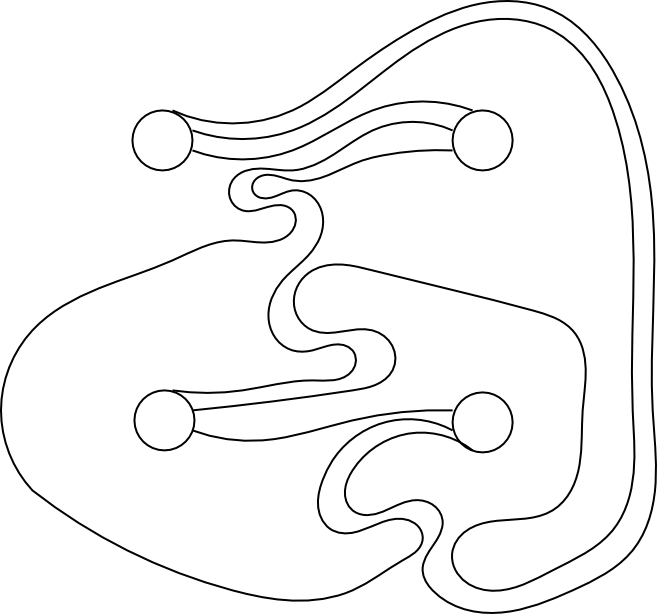}
\caption{\label{fig:curve} The generators $a, b$ and $c$ and the corresponding set of disjoint curves on the 4-punctured sphere with $a, b, c$-coordinates $(3, 1, 2)$. The curve system has slope $1/2$ on the 4-punctured sphere. } 
\end{figure} 

Vertices and paths on $\mathcal{D}$ are defined as follows.
\begin{itemize}
\item There are three types of vertices: $\langle p/q \rangle$, $\langle p/q \rangle^{\circ}$, and $\langle 1/0 \rangle$, where $p/q \not= 1/0$ is an arbitrary irreducible fraction. A vertex labeled $\langle p/q \rangle$ has horizontal coordinate $(q-1)/q$ and vertical coordinate $p/q$. A vertex labeled $\langle p/q \rangle^{\circ}$ has horizontal coordinate $1$ and vertical coordinate $p/q$. The vertex labeled $\langle 1/0 \rangle$ has coordinate $(-1,0)$. 

\item There is a path in the plane between distinct vertices $\langle p/q \rangle$ and $\langle r/s \rangle$ if $|ps-qr| = 1$. The path is denoted by $\langle p/q \rangle \text{\pdash}\langle r/s \rangle$. \emph{Horizontal edges} $\langle p/q \rangle^{\circ} \text{\pdash} \langle p/q \rangle$ and \emph{vertical edges} $\langle z\rangle^{\circ} \text{\pdash} \langle z\pm 1 \rangle^{\circ}$ are also allowed. 
\end{itemize}

See Figure \ref{fig:edgepathcomplex}. 
\begin{figure}[ht]
\centering
\includegraphics[scale=.2]{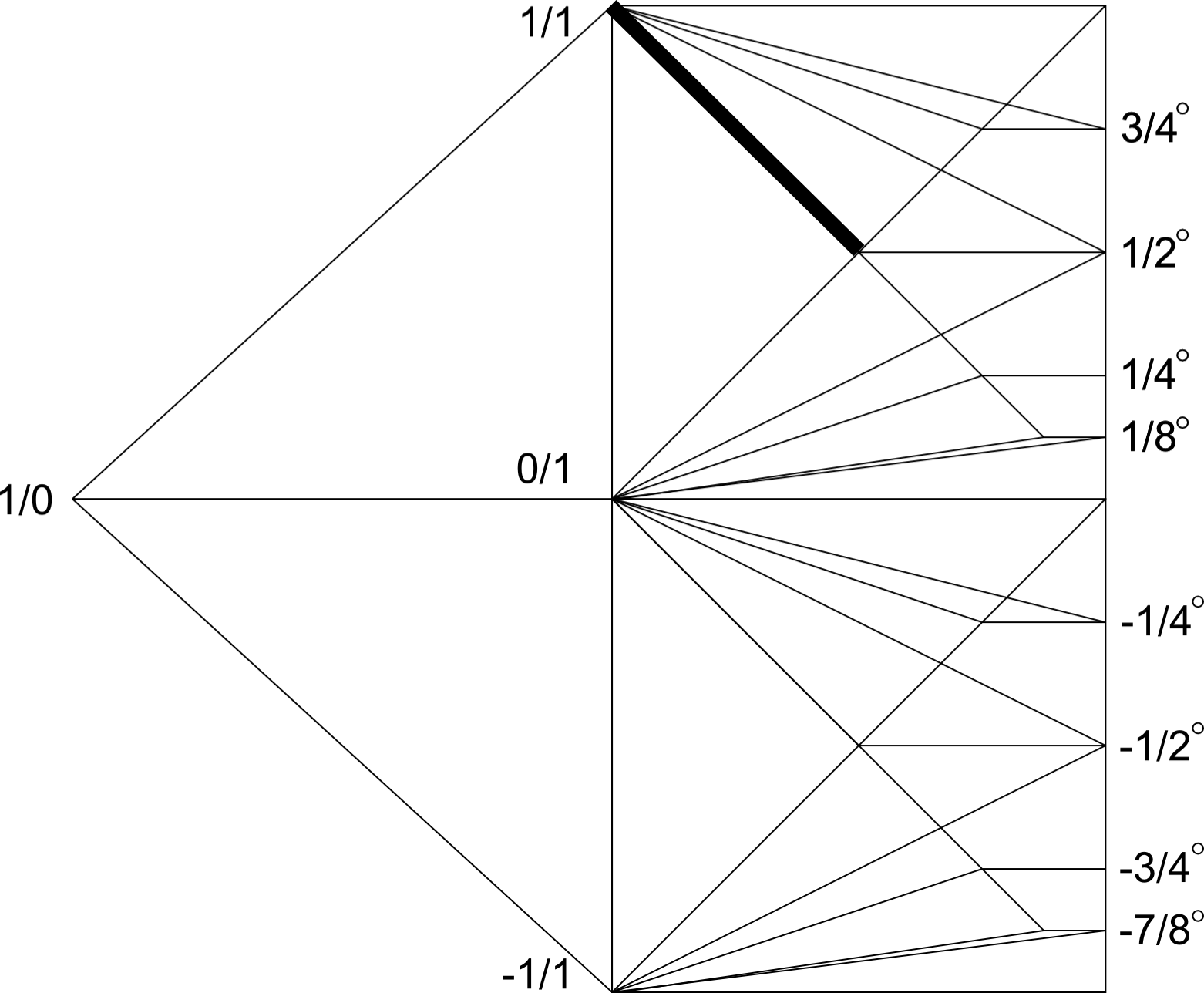}
\caption{A portion of the complex $\mathcal{D}$, with an edgepath from $1/1$ to $1/2$ indicated in bold.}
\label{fig:edgepathcomplex}
\end{figure}

\begin{defn} \label{defn:edgepath}
 A \emph{candidate edgepath} $\gamma$ for the fraction $p/q$ is a piecewise linear path in $\mathcal{D}$ satisfying the following properties: 
\begin{itemize}
\item[(E1)] The starting point of $\gamma$ lies on the edge $\langle p/q \rangle \text{\pdash} \langle p/q \rangle^{\circ}$. If the starting point is not the vertex $\langle p/q\rangle$ or $\langle p/q \rangle^{\circ}$, then $\gamma$ is constant. 
\item[(E2)] The edgepath $\gamma$ never stops and retraces itself, nor does it ever go along two sides of the same triangle in $\mathcal{D}$ in succession. 
\item[(E3)] The edgepath $\gamma$ proceeds monotonically from right to left, while motions along vertical edges are permitted. 
\end{itemize} 
An \emph{edgepath system} $\{\gamma_1, \ldots, \gamma_n\}$ is a set of edgepaths, one for each fraction $p_i/ q_i$ satisfying:
\begin{itemize}
\item[(E4)] \label{e:4} The endpoints of all the $\gamma_i$'s are points of $\mathcal{D}$ with identical $a, b$ coordinates and whose $c$-coordinate sum up to 0. In other words, the endpoints have the same horizontal coordinates and their vertical coordinates add up to zero. 
\end{itemize} 
\end{defn}
Additionally, if an edgepath ends at the point with slope $1/0$, then all other edgepaths in the system also have to end at the same point. 

\begin{thm}\cite[Proposition 1.1]{HatcherOertel} \label{thm:HO} Every essential surface in $S^3-K(\frac{p_1}{q_1},\ldots, \frac{p_n}{q_n})$ having nonempty boundary of finite slope is isotopic to one of the candidate surfaces. 
\end{thm} 

Based on Theorem \ref{thm:HO}, the Hatcher-Oertel algorithm enumerates all essential surfaces through the following steps. 
\begin{itemize}
\item For each fraction $\frac{p_i}{q_i}$, enumerate the possible edgepaths which correspond to continued fraction expansions of $\frac{p_i}{q_i}$ \cite{HT85}. 
\item Determine an edgepath system $\{\gamma_i \}$ by solving for sets of edgepath satisfying conditions (E1)-(E4). This gives the set of \emph{candidate surfaces}. 
\item Apply an incompressibility criterion in terms of edgepaths to determine if a given candidate surface is essential. 
\end{itemize} 

We describe these steps in detail in the next few sections.

\subsection{Applying the algorithm} 

We denote an edgepath by fractions $\langle \frac{p}{q} \rangle$,$\langle \frac{p}{q} \rangle^{\circ}$  and linear combinations of fractions connected by long dashes \pdash. The first fraction as we read from right to left will be written first. 

A point on an edge $\langle p/q \rangle \text{\pdash} \langle r/s \rangle$ is denoted by 
\[\frac{k}{m} \left \langle \frac{p}{q} \right \rangle + \frac{m-k}{m} \left \langle \frac{r}{s} \right \rangle, \]
with $a, b, c$-coordinates given by taking the linear combination of the $a, b, c$-coordinates of $\langle \frac{p}{q} \rangle$ and $\langle \frac{r}{s} \rangle$. 
\[ k(1, q-1, p) + (m-k)(1, s-1, r). \] This can then be converted to horizontal and vertical coordinates on $\mathcal{D}$.

We describe how to associate a candidate surface to a given edgepath system. Since we can isotope an essential surface so that if one edge of its edgepath is constant, then the entire edgepath is a single constant edge, we will only deal with the following two cases. 
\begin{itemize}
\item When $\gamma_i$ is constant.  \ \\
In this case, $\gamma_i$ is a single edge \[ \left \langle \frac{p}{q} \right\rangle^{\circ}\text{\pdash}\ \frac{k}{m}\left\langle \frac{p}{q} \right\rangle^{\circ} +  \frac{m-k}{m}\left\langle \frac{p}{q}\right \rangle.\]
Let $0 < \ell \leq 1$. We associate to $\gamma_i$ the surface in $B_i$ which has $2k$ arcs of slope $\frac{p}{q}$ coming into each pair of punctures of $S_i^2 \times \ell \setminus K$, and $m-k$ circles encircling a pair of punctures with slope $\frac{p}{q}$. Finally, we cap off the $m-k$ circles at $S_i^2 \times 0$.  

\item When $\gamma_i$ is not constant. 

Then $\gamma_i$ consists of edges of the form 
\[ \left \langle \frac{p}{q} \right\rangle \text{\pdash}\ \frac{k}{m}\left\langle \frac{p}{q} \right\rangle +  \frac{m-k}{m}\left\langle \frac{r}{s}\right \rangle.\] 
It begins with the vertex $\langle \frac{p_i}{q_i} \rangle$, and ends at $k/m\left\langle \frac{p}{q} \right\rangle + (m-k)/m\left\langle \frac{r}{s} \right \rangle$ for some fractions $\frac{p}{q}$, $\frac{r}{s}$. We associate to $\gamma_i$ the surface such that $S\cap S^2_i \times 0$ consists of $2m$ arcs going into a pair of punctures with slope $\frac{p_i}{q_i}$. For each successive edge in $\gamma_i$ of the form described above, we assign the surface whose intersection with $S_i^2\times \ell$ changes from $2m$ arcs of slope $\frac{p}{q}$ going into two pairs of punctures, to $2k$ arcs of slope $\frac{p}{q}$ going into the original pair of punctures and  $2(m-k)$ arcs of slope $\frac{r}{s}$ going into the other pairs of punctures through successive saddles. There is a choice, up to isotopy, of two possible slope-changing saddles, however, the choice does not affect the resulting homology class in $H_1(\partial N(K))$ of the boundary of the surface or its Euler characteristic. 
\end{itemize} 
See Figure \ref{fig:saddle} for these two cases. 
\begin{figure}[ht]
\includegraphics[scale=.19]{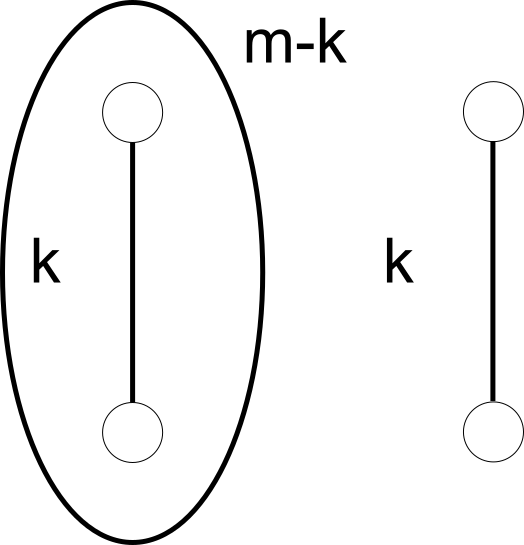} \ \\ 
\vspace{1cm}
\includegraphics[scale=.19]{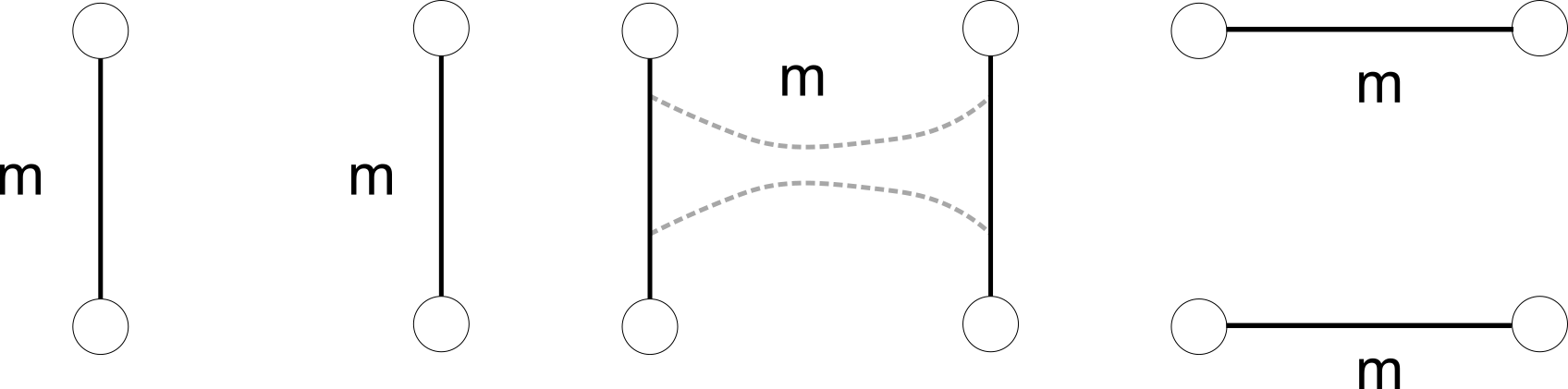}
\caption{\label{fig:saddle} Intersections $S\cap S_i^2\times \ell$: on top, a constant edgepath; below, a non-constant edgepath, see also Figure \ref{fig:twist}.  }
\end{figure}

To finish constructing the surface, we identify $2a$ half arcs and $b$ half circles on each of the two hemispheres and on the resulting single hemisphere. See Figure \ref{fig:hi}. 
\begin{figure}[ht]
\includegraphics[scale=1]{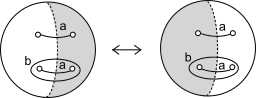}
\caption{ The two hemispheres that are identified.\label{fig:hi}}
\end{figure}

To check that a given candidate surface is essential, Hatcher and Oertel used a technical idea of the \emph{$r$-values} of the edgepath system (not the same $r$ as in $P(\frac{1}{r}, \frac{1}{s}, \frac{1}{t}$)). The idea is to examine the intersection of a compressing or $\partial$-compressing disk with the boundary sphere of each ball $B_i$, which will determine an $r$-value for each edgepath $\gamma_i$. If the $r$-values of a candidate surface disagree with the values that would result from the existence of a compressing disk, then it is incompressible. We state the criterion for incompressibility in terms of quantities that are easily computed given an edgepath system. 

\begin{defn} The \emph{r-value} for an edge $\langle \frac{p}{q} \rangle \text{\pdash} \langle \frac{r}{s} \rangle$ where $\frac{p}{q} \not= \frac{r}{s}$ is $s-q$. If $\frac{p}{q}= \frac{r}{s}$ or the path is vertical, then the $r$-value is 0. 
\end{defn} 
The $r$-value for an edgepath $\gamma$ is just the $r$-value of the final edge of $\gamma$. 

\begin{thm} \label{thm:esscriterion} \cite[Corollary 2.4]{HatcherOertel} A candidate surface is incompressible unless the cycle of $r$-values of $\{\gamma_i\}$  is one of the following types: 
\begin{itemize}
\item $(0, r_2,\ldots, r_n)$ 
\item $(1, \ldots, 1, r_n)$. 
\end{itemize} 
\end{thm} 

\subsection{Computing boundary slope from an edgepath system.} \label{subsec:slope}

Given an edgepath system $\{\gamma_i\}$ corresponding to an essential candidate surface, we describe how to compute its boundary slope. Note that there may be infinitely many surfaces carried by an edgepath system, however, they all have a common boundary slope. We use a representative with the minimum number of sheets to make our computations. Within a ball $B_i$, all surfaces look alike near $S^2_0$, thus we need only to consider the contribution to the boundary slope from the rest of the surface. The number of times the boundary of the surface winds around the longitude is given by $m$, the number of sheets of the surface. We measure the twisting around the meridian by measuring the rotation of the inward normal vector of the surface. Each time the surface passes through a non-constant saddle which does not end at arcs of slope 1/0, the vector goes through two full rotations. We choose the counterclockwise direction (and therefore the direction for a slope-increasing saddle) to be negative, and we choose the clockwise direction to be positive. See Figure \ref{fig:twist}. We do not deal with the case where the saddle ends at arcs of slope 1/0 in this paper, but it is easy to see that these saddles do not contribute to boundary slope.

\begin{figure}[ht]
\includegraphics[scale=.5]{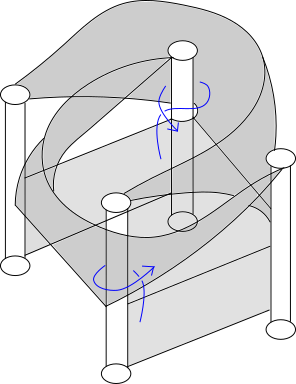}
\caption{A saddle going from arcs of slope 0/1 to arcs of slope 1/2 is shown in the picture. Note that on each of a pair of opposite punctures, the inward-pointing normal vector of the surface twists through arcs of slope 1/0 once. \label{fig:twist}}
\end{figure} 

The total number of twists $\tau(S)$ for a candidate surface $S$ from $\ell=0$ to $\ell=1$ is defined as
\[ \tau(S) := 2(s_- - s_+)/m = 2(e_- - e_+), \] where $s_-$ is the number of slope-decreasing saddles and $s_+$ is the number of slope-increasing saddles of $S$. This measures the contribution to the boundary slope of $S$ away from $\ell =0$.
In terms of egdepaths, $\tau(S)$ can be written in terms of the number $e_-$ of edges of $\gamma_i$ that decreases slope and $e_+$, the number of edges that increases slope. For an interpretation of this twist number in terms of lifts of these arcs in $\mathbb{R}^2 \setminus \mathbb{Z} \times \mathbb{Z}$, see \cite[Pg. 460]{HatcherOertel}.  If $\gamma_i$ ends with the segment 
\[ \left \langle \frac{p}{q} \right \rangle \text{\pdash}\ \frac{k}{m} \left \langle \frac{p}{q} \right \rangle + \frac{m-k}{m} \left  \langle \frac{r}{s} \right \rangle, \] 
then the final edge is counted as a fraction $1-k/m$. 
We add back the twists in the surface at level $\ell=0$ by subtracting the twist number of a Seifert surface $S_0$ obtained from the algorithm. The reason for this is that a Seifert surface always has zero boundary slope. Finally, the boundary slope of a candidate surface $S$ is 
\[ b_s = \tau(S) - \tau(S_0). \] 

In the interest of brevity, we do not discuss how to find this Seifert surface and merely exhibit examples. For a general algorithm to determine a Seifert surface which is a candidate surface, see the discussion in \cite[Pg. 460]{HatcherOertel}. 

\subsection{Computing the Euler characteristic from an edgepath system.} \label{subsec:Euler}
From the construction of Section \ref{subsec:surface}, we compute the Euler characteristic of a candidate surface associated to an edgepath system $\{\gamma_i\}$, where none of the $\gamma_i$ are constant or ends in $1/0$ as follows. Recall that $m$ is the number of sheets of the surface $S$. We begin with $2m$ disks of slope $\frac{p_i}{q_i}$ in each $B_i$.  

\begin{itemize}
\item Each non-fractional edge $\langle \frac{p}{q} \rangle \text{\pdash} \langle \frac{r}{s} \rangle$ is constructed by gluing $m$ number of saddles that changes $m$ arcs of slope $\frac{p}{q}$ to slope $\frac{r}{s}$, therefore decreasing the Euler characteristic by $m$. 
\item A fractional edge of the form $\langle \frac{p}{q} \rangle$\pdash $\frac{k}{m}\langle \frac{p}{q} \rangle + \frac{m-k}{m}\langle \frac{r}{s} \rangle$ changes $2(m-k)$ out of $2m$ arcs of slope $\frac{p}{q}$ to $2(m-k)$ arcs of slope $\frac{r}{s}$ via $m-k$ saddles, thereby decreasing the Euler characteristic by $m-k$. 
\end{itemize} 

This takes care of the individual contribution of an edgepath $\gamma_i$. Now the identification of the surfaces on each of the 4-punctured sphere will also affect the Euler characteristic of the resulting surface. In terms of the common $(a, b, c)$-coordinates shared by each edgepath,  there are two cases: 
\begin{itemize}
\item The identification of hemispheres between neighboring balls $B_i$  and $B_{i+1}$ identifies $2a$ arcs and $b$ half circles.  Thus it subtracts $2a+b$ from the Euler characteristic.  The final step of identifying hemispheres from $B_0$ and $B_n$ on a single sphere adds $b$ to the Euler characteristic. 
\end{itemize}

\section{proof of Theorem \ref{thm:slope2}} \label{sec:verifying}

We shall restrict the Hatcher-Oertel algorithm to $P(\frac{1}{r}, \frac{1}{s}, \frac{1}{t})$, when $r<0 $ and $s, t>0$ are odd. For 3-string pretzel knots, it is not necessary to include edges ending at the point with slope $\frac{1}{0}$ by the remark following Proposition 1.1 in \cite{HatcherOertel}. An edgepath system with endpoints at $\frac{1}{0}$ implies the existence of axis-parallel annuli in the surface, which either produce compressible components or can be eliminated by isotopy. 

For each fraction of the form $\frac{1}{p}$, there are two choices of edgepaths that satisfy conditions (E1) through (E3). They correspond to  two continued fraction expansions of $\frac{1}{p}$: 
\begin{align*}
\frac{1}{p}  &= 0 + [p]   \text{ gives edgepath } \left \langle \frac{1}{p} \right \rangle  \text{\pdash}\ \langle 0 \rangle,
\intertext{and} 
\frac{1}{p} &= \pm1 + \underbrace{[\pm 2, \pm 2, \ldots, \pm 2]}_{\text{$p-1$ times}} \text{ gives edgepath } \left \langle \frac{1}{p} \right \rangle \text{\pdash} \left \langle \frac{1}{p\pm1}\right \rangle\text{\pdash}\cdots\text{\pdash}\left \langle \mp 1 \right \rangle.   
\end{align*}
where it is a minus or a plus sign for the slope of each vertex for the second type of continued fraction expansion if $p$ is positive or negative, respectively.  

To show Theorem \ref{thm:slope2}, we exhibit edgepath systems satisfying conditions (E1)-(E4) in Definition \ref{defn:edgepath} corresponding to essential surfaces in the complement of $P(\frac{1}{r}, \frac{1}{s}, \frac{1}{t})$. We compute their boundary slopes and Euler characteristics using the methods of Section \ref{subsec:slope} and Section \ref{subsec:Euler}, respectively. For Conjecture \ref{conj:slopes}b, note that if an essential surface $S$ has boundary slope $x_j/y_j$ where $(x_j, y_j) = 1$, then $y_j$ is the minimum number of intersections of a boundary component of $S$ with a small meridian disc of $K$. Therefore, the number of sheets $m$ is given by $m = |\partial S|y_j,$ and we have 
\[ \frac{\chi(S)}{|\partial S|y_j} = \frac{\chi(S)}{m} \] as in the conjecture. Therefore, we need only to exhibit an essential surface for which
\[ \frac{\chi(S)}{m} = b_j. \]

\begin{proof} (of Theorem \ref{thm:slope2})\ 
Note that the boundary slope of a candidate surface $S$ corresponding to an edgepath system is given by $\tau(S) - \tau(S_0)$, where $S_0$ is a candidate surface that is a Seifert surface, see Section \ref{subsec:slope}. When all of $r,  s, t$ are odd, there is only one choice of edgepath system that will give us an orientable spanning surface \cite[Pg. 460]{HatcherOertel}. In this case, the edgepath system for $S_0$ is the following. 
\begin{itemize}
\item For $\frac{1}{r}$: $\langle \frac{1}{r} \rangle \text{\pdash} \langle 0 \rangle,$
\item For $\frac{1}{s}$:  $\langle \frac{1}{s} \rangle \text{\pdash} \langle 0 \rangle,$ 
\item For $\frac{1}{t}$:  $\langle \frac{1}{t} \rangle \text{\pdash} \langle 0 \rangle$. 
\end{itemize} 
Therefore, $\tau(S_0) = 2$. 

\begin{itemize}
\item[(1)] The case $2|r| < s, t$. 
\noindent \paragraph{\textbf{Boundary slope.}}
This will just be the same edgepath system as the Seifert surface and hence the boundary slope is $\tau(S) - \tau(S_0) = 0$. See Figure \ref{fig:state} for a picture. 

\begin{figure}[ht]
\includegraphics[scale=.2]{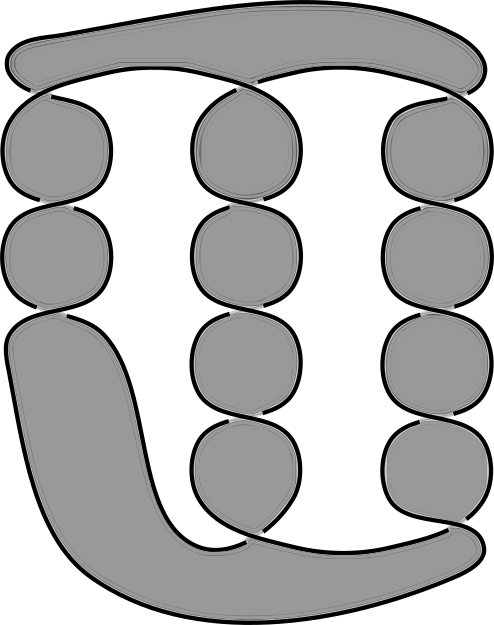}
\caption{\label{fig:state}An example for $P(-\frac{1}{3}, \frac{1}{5}, \frac{1}{5})$. The essential surface with boundary slope the Jones slope is the state surface obtained by taking the $B$-resolution for all crossings in the first twist region and the $A$-resolution for all twists in the second and the third region.}
\end{figure} 
\noindent \paragraph{\textbf{Euler characteristic.}}
It is clear that 
\begin{align*}
\frac{\chi(S)}{m} = -1.
\end{align*}

\item[(2)] The case $|r| > s$ or $|r| > t$.
\noindent \paragraph{\textbf{Boundary slope.}} We consider the following edgepath system.
\begin{itemize}
\item For $1/r$:$ \left \langle \frac{1}{r} \right \rangle \text{\pdash} \left \langle \frac{1}{r+1}  \right \rangle \text{\pdash} \cdots \text{\pdash} \left \langle -1 \right \rangle.$
\item For $1/s$: 
$\left \langle \frac{1}{s} \right \rangle \text{\pdash}\left \langle 0 \right \rangle. $
\item For $1/t$: 
$ \left \langle \frac{1}{t} \right \rangle \text{\pdash}\left \langle 0 \right \rangle. $
\end{itemize}
Condition (E4) requires that we set the $a, b$-coordinates for $\langle \frac{-1}{q+1} \rangle \text{\pdash} \langle \frac{-1}{q} \rangle$ for $0 < q \leq |r|$, and $\langle \frac{1}{s} \rangle \text{\pdash} \langle 0 \rangle$, and $\langle \frac{1}{t} \rangle \text{\pdash} \langle 0 \rangle$ are equal, and that the $c$-coordinates add up to zero. This is equivalent to setting the horizontal coordinates equal and summing the vertical coordinates to zero. We get the following equations:
\begin{align*}
& \frac{m(q-1)+k}{mq + k} = \frac{k'(s-1)}{m+k'(s-1)} =  \frac{k'(t-1)}{m+k'(t-1)} \notag \\ 
& \frac{-m}{mq+k} + \frac{k'}{m+k'(s-1)} + \frac{k''}{m+k''(t-1)} = 0.
\intertext{Recall that for the curve system represented by each point, the number $k$ represents the number of arcs coming into each puncture with one slope and $m-k$ represents the number of arcs coming into each puncture of a different slope. The number of sheets $m$ will be the same. We set $A = \frac{k}{m}, B = \frac{k'}{m}$, and $C = \frac{k''}{m}$ and solve.} 
B &= \frac{t-1}{-2+s+t}.
\intertext{Note the appearance of the quantity $\frac{t-1}{-2+s+t}$ also in the computation of the maximal degree of the colored Jones polynomial at the end of the proof of Theorem \ref{thm.JonesDeg}. It showed up as the $N$-dependent part of the maximum of the quadratic on the boundary of the summation range. 
To compute $\tau(S)$, note that the edges are all decreasing. We add up $A, B, C$, and the number of paths for $\tau(S)$.}
\tau(S)&= 2(-r-q-A +1- B +1- C).
\end{align*}
The boundary slope is then
\[\tau(S) -\tau(S_0) = \frac{2(1-st)}{-2+s+t} -2r.  \] 

\noindent \paragraph{\textbf{Euler characteristic}}
We will now compute the Euler characteristic for this representative  of the edgepath system. For each of  the edgepaths we first have $ 3\cdot 2m$ number of base disks with slopes the slopes of the tangles corresponding to $\{\gamma_i\}$, then we glue on saddles.  The sum total of the change in Euler characteristic after constructing the surface according to these local edgepaths is then \[-m \cdot (-r-q-1) - m(1-A+1-B+1-C).\] This also accounts for the contribution of the fractional last edge of each of the edgepaths.

Then we consider the contribution to the Euler characteristic from gluing these local candidate surfaces together, which in terms of $(a, b, c)$ coordinates will be \[ -2(2a+b) + b. \]  We use the third edgepath 
\[\left \langle \frac{1}{t} \right \rangle \text{\pdash}\ \frac{k''}{m}\left \langle \frac{1}{t} \right \rangle + \frac{m-k''}{m} \left \langle 0 \right \rangle \]
to compute $a$ and $b$ in terms of $r, s$, and $t$.
Adding everything together and dividing by the number of sheets, we have
\begin{align*}
\frac{\chi(S)}{m} &= 2+r.
\end{align*}  
Finally, the candidate surfaces corresponding to the two types of edgepath systems exhibited above are all essential by Theorem \ref{thm:esscriterion}, since the $r$-values are of the form $(r-1, s-1, t-1)$ or $(1, s-1, t-1)$ and it follows from our assumptions that $|r|, s, t > 2$. 
\end{itemize}
\end{proof}

\section{Further directions.} \label{sec:discussion}

For general Montesinos knots of arbitrary length $K(\frac{p_1}{q_1}, \ldots, \frac{p_n}{q_n})$, the techniques used in this paper will not easily apply due to computational complexity. As discussed in Section \ref{subsec:main}, we need only to consider the case where the first tangle is negative as the rest of the Montesinos knots will be adequate. In a forth-coming paper \cite{LV}, we will discuss possible extensions of Theorem \ref{thm:slope1} and Theorem \ref{thm:slope2} using different techniques from that of this paper. In particular, let $P(\frac{1}{r}, \frac{1}{s_1}, \ldots, \frac{1}{s_{n-1}})$ be a pretzel knot where $r < 0 < s_i$ are odd for $1\leq i \leq n-1$. We are able to show that if $2|r|< s_i$ for all $i$, then the Jones slope is matched by the boundary slope of a state surface. We are also able to obtain statements similar to case (2) of Theorem \ref{thm:slope1} when $|r| > s_i$ for some $1 \leq i \leq n-1$. It is more challenging to generalize the expression for the constant terms given in the theorem, since in this case their topological meaning is not yet clear. We hope to clarify this in the future.

\bibliographystyle{amsalpha}\bibliography{references}

\end{document}